\documentclass[12pt]{amsart}
\usepackage{amsmath,amsthm,latexsym,amscd,amsbsy,amssymb,url}
\usepackage[all]{xypic}
 \relax

\chardef\bslash=`\\ % p. 424, TeXbook

\makeatletter
\def\verbatim{\interlinepenalty\@M \@verbatim
  \leftskip\@totalleftmargin\advance\leftskip2pc
  \frenchspacing\@vobeyspaces \@xverbatim}
\makeatother
\newtheorem{thm}{Theorem}[section]
\newtheorem{cor}[thm]{Corollary}
\newtheorem{lem}[thm]{Lemma}
\newtheorem{pro}[thm]{Proposition}

\newtheorem{ex}[thm]{Example}
\newtheorem{que}[thm]{Question}

\begin{document}

%%%%%%% Begin Topmatter %%%%%%%%%%

\title
{A Note On Separating Function Sets}
\author{Raushan Buzyakova}
\email{Raushan\_Buzyakova@yahoo.com}
\author{Oleg Okunev}
\email{oleg@servidor.unam.mx}
\address{Facultad de Ciencias Físico-Matematicas, Benemérita Universidad Autonoma de Puebla, Apdo postal 1152, Puebla, Puebla CP 72000, Mexico}
\keywords{$C_p(X)$, discrete space, point-separating set, spread, $i$-weight, $\sigma$-product}
\subjclass{54C35, 54E45, 54A25}

%%%%%%% End topmatter %%%%%%%%%

\begin{abstract}{
We study separating function sets. We find some necessary and sufficient conditions for $C_p(X)$ or $C_p^2(X)$ to have a point-separating subspace that is a metric space with certain nice properties. One of the corollaries to our discussion is that for a zero-dimensional $X$, $C_p(X)$ has a discrete point-separating space if and only if $C_p^2(X)$ does.
}
\end{abstract}

\maketitle
\markboth{R. Buzyakova and O. Okunev}{A Note On Separating Function Sets}
{ }

\par\bigskip
\section{Introduction}\label{S:introduction}

\par\bigskip
To start our discussion, recall, that $F\subset C^n_p(X)$ is {\it point-separating} if for any distinct $x,y\in X$ there exists $\langle f_1,...,f_n\rangle\in F$ such that
$f_i(x)\not = f_i(y)$ for some $i\leq n$. 
In this paper we are concerned with the following  general problem.
\par\bigskip\noindent
{\bf Problem.} {\it
Let $P$ be a nice property. Describe "$C_p(X)$ (or $C_p^m(X)$) having a point-separating subspace with  $P$"  in terms of the topology of $X$, $X^n$, or $X^\omega$.
}
\par\bigskip\noindent
In this study, $P$ is  the property of being a discrete space, a countable union of discrete subspaces, a metric compactum, or a discrete group. 
We obtain two characterizations of  spaces $X$ for which  $C^2_p(X)$ has a discrete point-separation subspace (Theorems \ref{thm:criterion} and \ref{thm:criterionGCH}, and \ref{thm:criterionsigma}). One of the characterizations  is consistent and may have  a chance for a ZFC proof. We also characterize zero-dmensional $Z$ with point-separating discrete subspaces in $C_p(X)$  (Theorems \ref{thm:criterion0dim} and \ref{thm:criterionGCH0dim}, and \ref{thm:criterionsigma0dim}). Questions of similar nature are quite popular among topologists interested in $C_p$-theory and have been considered in many papers.

In notation and terminology we follow \cite{Eng}. All spaces under consideration are assumed Tychonoff and {\it  infinite}.   By $s(X)$ we denote the supremum of cardinalities of discrete subspaces of $X$. By $iw(X)$ we denote the smallest  weight of a Tychonoff subtopology of $X$. When we say that $D$ is a discrete subspace of $X$, $D$ need not be closed in $X$. By $\sigma_X(x^*)$ we denote the subspace of $X^\omega$ that consists of all points that differ from $x^*$ by finitely many coordinates. Since $\sigma_X(x)$ and $\sigma_X(y)$ are obviously homeomorphic we may simply write $\sigma_X$ and, as usual,  refer to it as {\it $\sigma$-product of $X^\omega$}. A standard neighborhood of $f$ in $C_p(X)$ is in form $U(x_1,...,x_n; B_1,...,B_n)=\{g\in C_p(X): g(x_i)\in B_i\}$, where $B_i$ is open interval with rational endpoints for each $i$. Zero-dimensionality is understood in the sense of dim.

\par\bigskip
\section{Discrete Point-Separating Subspaces}\label{S:discretesubspaces}

\par\bigskip
Our first goal is to find a characterization of those infinite $X$ for which $C_p(X)$ or $C_p^2(X)$ has a point-separating discrete subspace. We  achieve the goal within a wide class of spaces. We start with a few helping lemmas.

The following Lemma is almost identical to Proposition II.5.5 in \cite{Arh} but due to cofinality restrictions we have to prove it using a similar argument. 
\par\bigskip\noindent
\begin{lem}\label{lem:functionspace0} (version of \cite[II.5.5]{Arh})
Assume that $C^2_p(X)$ has a discrete  subspace of size $\tau$. Then the following hold:
\begin{enumerate}
	\item If  $cf(\tau)>\omega$, then
$s(X^n)\geq \tau$ for some $n\in \omega$.
	\item If $cf(\tau)=\omega$, then $s(\sigma_X)\geq \tau$.
\end{enumerate}
\end{lem}
\begin{proof}
Since part (2)  is an obvious consequence of part (1), we will prove part (1) only.
Let $D\subset C^2_p(X)$ be a $\tau$-sized discrete subspace.
For each $\langle f, g\rangle\in D$ fix standard neighborhoods
$U_f = U_f(x^f_1,...,x^f_{n_f}, I_1^f, ...,I_{n_f}^f)$ and 
$V_g = V_g(y^g_1,...,y^g_{m_g}, J_1^g, ...,J_{m_g}^g)$  such that $U_f\times V_g$ contains $\langle f,g\rangle$ and misses
 $D\setminus \{\langle f,g\rangle\}$.
Since $cf(\tau ) >\omega$ we can find $n^*,m^*\in \omega$, $\langle I_i: i\leq n^*\rangle$,
$\langle J_i: i\leq m^*\rangle$, and a $\tau$-sized $D'\subset D$  such that $n_f=n^*$, $m_g=m^*$,
$\langle I_i^f: i\leq n^*\rangle = \langle I_i: i\leq n^*\rangle$, and $\langle J_i^g: i\leq m^*\rangle = \langle J_i: i\leq m^*\rangle$  for each $\langle f,g\rangle\in D'$. 
We can now conclude that for any distinct $\langle f,g\rangle,\ \langle f',g'\rangle\in D'$, either   
$\langle x_i^f:i\leq n^*\rangle \not = \langle x^{f'}_i:i\leq n^*\rangle$ or
$\langle y_i^g:i\leq m^*\rangle \not = \langle y^{g'}_i:i\leq m^*\rangle$. Therefore, the set
$S=\{\langle x^f_i,...,x_{n^*}^f, y_i^g,...,y^g_{m^*}  \rangle: \langle f,g\rangle \in D' \}$ is $\tau$-sized.
To show that $S$ is a discrete subspace of 
$X^{n^*+m^*}$,  for each $\langle f,g\rangle\in D'$, 
put $U_f = f^{-1}(I_1)\times ...\times f^{-1}(I_{n^*})$ and $V_g = g^{-1}(J_1)\times ...\times g^{-1}(J_{m^*})$. Clearly $U_f\times V_g$ is a neighborhood of $\langle x^f_i,...,x_{n^*}^f, y_i^g,...,y^g_{m^*}  \rangle$ in $X^{n^*+m^*}$. Next, fix $\langle f',g'\rangle\in D'\setminus \{\langle f,g\rangle\}$. By the choice of our neighborhoods, we may assume that $f\not \in U_{f'}$.
Therefore, there exists $i\leq n^*$ such that $f(x^{f'}_i)\not \in I_i$. Therefore, $x_i^{f'}\not \in f^{-1}(I_i)$, which implies that $\langle x^{f'}_i,...,x_{n^*}^{f'}, y_i^{g'},...,y^{g'}_{m^*}  \rangle\not \in U_f\times V_g$.
\end{proof}

\par\bigskip
Note that if $C_p(X)$ or $C^2_p(X)$ has a discrete point-separating subspace of an infinite size $\tau$ , then $\tau \geq iw(X)$. If in addition $cf(\tau)>\omega$, then, by Lemma \ref{lem:functionspace0}, $s(X^n)\geq \tau \geq iw(X)$ for some $n$. Thus, the following statement holds.

\par\bigskip\noindent
\begin{thm}\label{thm:functionspace}
Assume that $C_p(X)$ or $C_p^2(X)$  has a discrete point-separating subspace of size $\tau$ with $cf(\tau)>\omega$. Then
$s(X^n)\geq \tau\geq iw(X)$ for some natural number $n$.
\end{thm}

\par\bigskip\noindent
We are now ready to formulate and prove two necessary conditions for $C_p(X)$ and $C_p^2(X)$  to have a point-separating discrete subspace.
\par\bigskip\noindent
\begin{thm}\label{thm:functionspace2}
If $C^2_p(X)$ has a discrete point-separating subspace, then 
$s(\sigma_X)\geq iw(X)$.
\end{thm}
\begin{proof}
Put $\tau = iw(X)$. If $\tau$ is countable, then $X$ has a countable network. Since $X$ is infinite, it contains an infinite countable subspace. Hence, $s(\sigma_X)\geq iw(X)$. 

We now assume that $\tau$ is uncountable. By Theorem \ref{thm:functionspace} we may assume that $cf(\tau) = \omega$. Fix a strictly increasing sequence of cardinals $\tau_n$ of uncountable cofinalities so that $\tau =\sum_n\tau_n$. Since any point-separating subset of $C^2_p(X)$ must have size at least $\tau$, there exists a discrete subset of cardinality $\tau_n$ in $C^2_p(X)$ for each $n$. By
Lemma \ref{lem:functionspace0}, there exists a discrete subset $D_n$ in some finite power of $X$ for each $n$.
Therefore, $s(\sigma_X)\geq \tau$.
\end{proof}

\par\bigskip\noindent
In all future arguments, the cases when a discrete point-separating subspace is finite can be handled as in 
Theorem \ref{thm:functionspace2} and will therefore not be considered.
For our next observation we need Zenor's theorem \cite{Zen} stating that if $s(X\times Y)\leq \tau\geq \omega$ then
either $hl(X)\leq \tau$ or $hd(Y)\leq \tau$.

\par\bigskip\noindent
\begin{thm}\label{thm:functionspace3}
Assume Generalized Continuum Hypothesis. If $C^2_p(X)$ has a discrete point-separating subspace, then 
$s(X^n)\geq iw(X)$ for some $n\in \omega$.
\end{thm}
\begin{proof}
Put $\tau = iw(X)$.
By Theorem \ref{thm:functionspace} we may assume that $\tau$ is an infinite cardinal of countable cofinality. Assume the contrary. Then $s(X^4) = \lambda   < \tau$. By Zenor's theorem, $hl(X^2)\leq \lambda$ or $d(X^2)\leq \lambda$. If the former is the case, then the off-diagonal part of $X^2$ can be covered by $\lambda$-many functionally closed boxes, which implies that $iw(X)< \tau$. If $d(X^2)\leq \lambda$, then by Generalized Continuum Hypothesis, $w(X^2)$ is at most $2^\lambda<\tau$. Since both cases lead to contradictions, the statement is proved.
\end{proof}

\par\bigskip\noindent
The assumptions in Theorem \ref{thm:functionspace3} prompts the following questions.
\par\bigskip\noindent
\begin{que}\label{que:zfc}
Does Theorem \ref{thm:functionspace3} hold in ZFC?
\end{que}
\par\bigskip
Note that if one can construct a space $X$ such that $s(X^n)=\omega_n$ for all natural numbers $n$ and $iw(X)=\omega_\omega$, then the answer to Question \ref{que:zfc} is a "no".

At this point one may wonder if our study is justified. In other words, are we studying a non-empty class? Let $X$ be an  non-metrizable  compact space such  that $X^n$ is hereditary separable for each $n$. Such a space exists. A consistent example of such a space is Ivanov's modification \cite{Iva} of Fedorchuk's example \cite{Fed}.  Since $X^n$ is hereditarily separable,
by Lemma \ref{lem:functionspace0}, no discrete subspace of $C_p^2(X)$ or $C_p(X)$ is uncountable.  Since $X$ is not submetrizable, we conclude that no countable subspace of $C_p^2(X)$ or $C_p(X)$ is point-separating. Let us summarize this observation as follows.
\par\bigskip\noindent
\begin{ex}\label{ex:sspace}
There exists a consistent example of a compactum $X$ such that neither $C_p(X)$ nor $C^2_p(X)$ has a discrete separating subspace.
\end{ex}

\par\bigskip\noindent
The authors believe that in some models of ZFC, no such example may exist, meaning that any space may have a discrete in itself point-separating function set. 

\par\bigskip\noindent
\begin{que}\label{que:mainexample}
Is there a ZFC example of a space $X$ such that no discrete subspace of $C_p(X)$ ($C_p^n(X)$) is point-separating?
\end{que}

We will next reverse the statement of Theorem \ref{thm:functionspace}, which will bring us to the promised characterizations.

\par\bigskip\noindent
\begin{thm}\label{thm:Xn}
If  $X^n$ has a discrete subspace of size $iw(X)$ for some natural number $n$, then  $C^2_p(X)$ has a point-separating discrete subspace.
\end{thm}
\begin{proof} Let $n$ be the smallest that satisfies the hypothesis of the lemma and  put $\tau = iw(X)$. By the choice of $n$ there exists a $\tau$-sized  discrete subspace $D$ of $X^n$ with the following property:
\par\medskip
{\it Property:} $|\{x(i):i\leq n\}|=n$ for each $x\in D$.
\par\medskip\noindent
Let $\mathcal T$ be a Tychonoff subtopology of the topology of $X$ of weight $\tau$. Fix a $\tau$-sized network $\mathcal N$ for $\langle X, \mathcal T\rangle$ that consists of functionally closed subsets. Let $\mathcal P$ be the set of all pairs $\langle A, B\rangle$ of disjoint elements of $\mathcal N$. Enumerate $\mathcal P$ and $D$ as $\{\langle A_\alpha, B_\alpha\rangle:\alpha<\tau\}$ and $\{d_\alpha:\alpha<\tau\}$. Since $D$ is a discrete subspace, for each $\alpha<\tau$ we can fix a functionally closed set  $B^\alpha_1\times ...\times B^\alpha_n$ that contains $d_\alpha$ in its interior and misses $D\setminus \{d_\alpha\}$. By {\it Property}, we may assume that $B^\alpha_i\cap B^\alpha_j=\emptyset$ for distinct $i$ and $j$.

We will next construct our desired subspace $\{\langle f_\alpha, g_\alpha\rangle:\alpha<\tau\}$ of $C_p(X)$. 
\par\medskip\noindent
\begin{description}
	\item[\underline{\it Definition of $f_\alpha$, where $\alpha<\tau$}]  Let $S_\alpha$ be a functionally closed subset of $X$ such that $X\setminus S_\alpha$ can be written as a  union of $L_\alpha$ and $R_\alpha$ so that the following hold.
\begin{enumerate}
	\item $cl_X(L_\alpha)\cap cl_X(R_\alpha)\subset S_\alpha$;
	\item $A_\alpha\subset L_\alpha$ and $B_\alpha\subset R_\alpha$;
	\item $d_\alpha(i)\in L_\alpha$ if $d_\alpha(i)\not \in B_\alpha$, and $d_\alpha(i)\in R_\alpha$ if $d_\alpha(i)\in B_\alpha$.
\end{enumerate}
Such an $S_\alpha$ exists since $A_\alpha$ and $B_\alpha$ are functionally separable sets and the coordinate set of $d_\alpha$ is finite.  Let $f_{\alpha, l}: L_\alpha\cup S_\alpha\to [-1, 0]$ be any continuous function that has the following properties:
\begin{description}
	\item[\rm L1] $f^{-1}_{\alpha, l}(\{0\})= S_\alpha$;
	\item[\rm L2] $(\{d_\alpha(i):i\leq n\}\cap L_\alpha )\subset  f^{-1}_{\alpha, l}([-1,-1/3))\subset \bigcup 
\{B^\alpha_i :d_\alpha(i)\in L_\alpha\}$.
\end{description}
Such a function exists because the coordinate set of $d_\alpha$ is finite and $S_\alpha$ is functionally closed and misses the coordinate set of $d_\alpha$.
 Let $f_{\alpha , r}: R_\alpha\cup S_\alpha\to [0, 1]$ be any continuous function that has the following properties:
\begin{description}
	\item[\rm R1] $f^{-1}_{\alpha, r}(\{0\})= S_\alpha$;
	\item[\rm R2] $(\{d_\alpha(i):i\leq n\}\cap  R_\alpha )\subset  f^{-1}_{\alpha, r}((1/3, 1])\subset \bigcup 
\{B^\alpha_i :d_\alpha(i)\in R_\alpha\}$.
\end{description}
Put $f_\alpha = f_{\alpha ,l} \cup f_{\alpha , r}$. By L1 and R1, $f_\alpha$ is a continuous function from $X$ to $\mathbb R$. 
\end{description}

\par\medskip\noindent
\begin{description}
	\item[\underline{\it Definition of $g_\alpha$, where $\alpha<\tau$}] Let  $g_\alpha$ be any continuous function
that maps $B_i^\alpha$ to $(i-1/3, i+1/3)$. This can be done since $B_i^\alpha$'s form a disjoint finite collection of functionally closed sets.
\end{description}
It remains to show that  $F=\{\langle f_\alpha, g_\alpha\rangle:\alpha<\tau\}$ is  a point-separating   discrete subspace.
To show that $F$ is point-separating, fix distinct $a,b$ in $X$. Since $\mathcal N$ is a network, there exist disjoint $A,B\in \mathcal N$ that contain $a$ and $b$, respectively. Then $\langle A,B\rangle=\langle A_\alpha, B_\alpha\rangle\in \mathcal P$.
By the definition of $f_\alpha$, $f_\alpha(a)=f_{\alpha, l}(a)<0$ and $f_\alpha(b)=f_{\alpha,r}(b)>0$.

\par\medskip\noindent
To show that $F$ is discrete in itself, fix $\alpha$. Put
$$
U_\alpha=\{f: f(d_\alpha(i)) <-1/3\ if \ d_\alpha(i)\in L_\alpha, f(d_\alpha(i))>1/3\ if \ d_\alpha(i)\in R_\alpha\}
$$
$$
V_\alpha = \{g: g(d_\alpha(i))\in (i-1/3,i+1/3)\}
$$
Clearly, $U_\alpha\times V_\alpha$ is a neighborhood of $\langle f_\alpha, g_\alpha\rangle$. To show that this neighborhood misses the rest of $F$, fix $\beta\not = \alpha$. There exists $i\leq n$ such that $d_\alpha(i)\not \in B^\beta_i$. We have two possible cases.
\begin{description}
	\item[\rm Case 1] This case's assumption is that $d_\alpha(i)\not \in \bigcup_{j\leq n}B^\beta_j$. By L2 and R2 of  the definition of $f_\beta$, we have $f_\beta((d_\alpha^i))\in (-1/3,1/3)$. Hence $f_\beta\not\in U_\alpha$. Therefore, $\langle f_\beta,g_\beta\rangle\not \in U_\alpha\times V_\alpha$.

	\item[\rm Case 2] Assume Case 1 does not take place. Then there exists $j\leq n$ such that $d_\alpha(i)\in B^\beta_j$. By the choice of $i$, we have  $i\not = j$. Therefore, $g_\beta(d_\alpha(i))\not \in (i-1/3, i+1/3)$. Hence $g_\beta\not\in V_\alpha$. Therefore, $\langle f_\beta,g_\beta\rangle\not \in U_\alpha\times V_\alpha$.

\end{description}
\end{proof}

\par\medskip\noindent
Statements \ref{thm:functionspace}, \ref{thm:Xn}, and  \ref{thm:functionspace3}  result in the following criteria.
\par\bigskip\noindent
\begin{thm}\label{thm:criterion}
Let a space $X$ have $iw(X)$ of uncountable cofinality. Then $C^2_p(X)$ has a point-separating discrete subspace if and only if $s(X^n)\geq iw(X)$ for some $n$.
\end{thm}

\par\bigskip\noindent
\begin{thm}\label{thm:criterionGCH}
Assume Generalized Continuum Hypothesis. Then $C^2_p(X)$ has a point-separating discrete subspace if and only if $s(X^n)\geq iw(X)$ for some $n$.
\end{thm}

\par\bigskip\noindent
Note that criteria \ref{thm:criterion} and \ref{thm:criterionGCH} would hold for $C_p(X)$ if we could prove Theorem   \ref{thm:Xn} for $C_p(X)$.

\par\bigskip\noindent
\begin{que}\label{que:CpX}
Assume that $X^n$ has a discrete subspace of size $iw(X)$ for some natural number $n$. Is it true that $C_p(X)$ has a discrete point-separating set?
\end{que}

\par\bigskip\noindent
Using an argument somewhat similar to that of Theorem \ref{thm:Xn} we will next show that Question \ref{que:CpX} has an affirmative answer if we assume that $C$ is zero-dimensional.

\par\bigskip\noindent
\begin{thm}\label{thm:Xnzerodim}
Assume that $X$ is zero-dimensional.
If  $X^n$ has  a discrete subspace of size $iw(X)$, then  $C_p(X)$ has a point-separating discrete subspace.
\end{thm}
\begin{proof} Let $n$ be the smallest that satisfies the hypothesis of the lemma and  put $\tau = iw(X)$. By the choice of $n$ there exists a $\tau$-sized  discrete subspace $D$ of $X^n$ with the following property:
 \par\medskip
{\it Property:} $|\{x(i):i\leq n\}|=n\ for\ each\ x\in D.$

\par\medskip\noindent
Let $\mathcal T$ be a Tychonoff subtopology of the topology of $X$ of weight $\tau$. Due to zero-dimensionality of $X$ and the factorization theorem of Mardesic \cite{Mar}, we may assume that $\mathcal T$ is zero-dimensional too.  Fix a $\tau$-sized network $\mathcal N$ for $\langle X, \mathcal T\rangle$ that consists of clopen  subsets. Let $\mathcal P$ be the set of all pairs $\langle A, B\rangle$ of disjoint elements of $\mathcal N$. Enumerate $\mathcal P$ and $D$ as $\{\langle A_\alpha, B_\alpha\rangle:\alpha<\tau\}$ and $\{d_\alpha:\alpha<\tau\}$.
We will next construct our desired subspace in $C_p(X)$.
\par\medskip\noindent
\underline{\it Definition of $f_\alpha$, where $\alpha<\tau$}]  Since $D$ is a discrete subspace, we can fix a clopen 
box $U^\alpha_1\times ...\times U^\alpha_n$ that contains $d_\alpha$ and misses $D\setminus \{d_\alpha\}$. 
By {\it Property}, we may assume that $U^\alpha_i\cap U_j^\alpha=\emptyset$ if $i\not = j$.
 Since $A_\alpha$ and $B_\alpha$ are disjoint, we may assume that each $U_i$ meets at most one of the sets $A_\alpha$ and $B_\alpha$. Define $f_\alpha: X\to \{0,1,2,...,n+1\}$ by letting $f_\alpha(U_i)=\{i\}$, $f_\alpha(A_\alpha\setminus \bigcup_{i\leq n}U_i)=\{0\}$, and $f_\alpha (X\setminus (A_\alpha\cup U_1\cup ...\cup U_n) )= \{n+1\}$.

\par\medskip\noindent
It remains to show that  $F=\{f_\alpha:\alpha<\tau\}$ is a  point-separating  discrete subspace.
To show that $F$ is point-separating, fix distinct $a,b$ in $X$. Since $\mathcal N$ is a network, there exist disjoint $A,B\in \mathcal N$ that contain $a$ and $b$, respectively. Then $\langle A,B\rangle=\langle A_\alpha, B_\alpha\rangle\in \mathcal P$.
Since no $U_i^\alpha$ meets both $A_\alpha$ and $B_\alpha$ at the same time,  $f_\alpha(A_\alpha)$ misses  $f_\alpha(B_\alpha)$.

\par\medskip\noindent
To show that $F$ is discrete in itself, fix $f_\alpha$ and put $V_\alpha = \{f: f(d_\alpha(i)) \in (i-1/3, i+1/3), i\leq n\}$. Next fix any $\beta\not = \alpha$. Then there exists $i\leq n$ such that $d_\alpha(i)\not \in U^\beta_i$. Therefore, $f_\beta(d_\alpha(i))\not \in (i-1/3,i+1/3)$. Hence, $f_\beta\not \in U_\alpha$.
\end{proof}

\par\medskip\noindent
Note that  Theorems \ref{thm:criterion} and  \ref{thm:criterionGCH} are now true for $C_p(X)$ provided $X$ is zero-dimensional. Let us state the new versions for reference.

\par\bigskip\noindent
\begin{thm}\label{thm:criterion0dim}
Let a zero-dimensional space $X$ have $iw(X)$ of uncountable cofinality. Then $C_p(X)$ has a point-separating discrete subspace if and only if $s(X^n)\geq iw(X)$ for some natural number $n$.
\end{thm}

\par\bigskip\noindent
\begin{thm}\label{thm:criterionGCH0dim}
Assume Generalized Continuum Hypothesis. Let $X$ be zero-dimensional. Then $C_p(X)$ has a point-separating discrete subspace if and only if $s(X^n)\geq iw(X)$ for some $n$.
\end{thm}

\par\bigskip\noindent
For our final characterization discussion we would like to extract a technical statement from the proof of Theorem \ref{thm:Xn} and prove one helpful proposition.
\par\bigskip\noindent
\begin{lem}\label{lem:extract}
Assume that a finite power of $X$ has a discrete subspace of size $\lambda$. Let $\{\langle A_\alpha, B_\alpha\rangle:\alpha<\lambda\}$ be a family of pairs of functionally closed disjoint subsets of $X$. Then there exists a discrete subspace $F$ in $C^2_p(X)$ with the following property:
\par\smallskip\noindent
(*) If $a\in A_\alpha$ and $b\in B_\alpha$ for some $\alpha<\lambda$, then $f(a)\not = f(b)$ for some $f\in F$.
\end{lem}

\par\bigskip\noindent
\begin{pro}\label{pro:union}
Let $C^m_p(X)$ contain a point-separating subspace which is a countable union of discrete subspaces. Then $C^m_p(X)$ has a discrete point-separating subspace.
\end{pro}
\begin{proof}
We will prove the statement for $m=2$.
Let $D=\cup_nD_n$ be a point-separating set of $C^2_p(X)$, where each $D_n$ is a discrete subspace. For each $n$, fix a homeomorphism $h_n:\mathbb R\to (n,n+1)$. Put $E_n = \{\langle h_n\circ f, h_n\circ g\rangle: \langle f,g\rangle \in D_n\}$. Clearly, $E_n$ separates $x$ and $y$ if and only if $D_n$ does. Also, $E_n$ is a discrete subspace of $C^2_p(X)$. Since all functions in $(\cup_kE_k)\setminus E_n$ target $\mathbb R\setminus (n,n+1)$, we conclude that the closure of $(\cup_kE_k)\setminus E_n$ misses $E_n$. Terefore, $\cup_n E_n$ is a point-separating discrete subspace of $C^2_p(X)$.
\end{proof}

\par\bigskip\noindent
\begin{thm}\label{thm:criterionsigma} 
$C^2_p(X)$ has a discrete point-separating subspace if and only if $s(\sigma_X) \geq iw(X)$.
\end{thm}
\begin{proof}
Necessity is done in Theorem \ref{thm:functionspace2}. To prove sufficiency, put $\tau = iw(X)$. Let $\mathcal N$ be a
$\tau$-sized family of functionally closed subsets of $X$ that is a network for some Tychonoff subtopology of $X$. Let $\mathcal P$ consist of all pairs of disjoint elements of $\mathcal N$. For each $n$ we can find a discrete subset $D_n$ of $\sigma_X$ that lives in a copy of some finite power of $X$ so that $\tau = \sum_n |D_n|$. Next represent $\mathcal P$ as $\bigcup \mathcal P_n$, where $|\mathcal P_n|=|D_n|$. Applying Lemma \ref{lem:extract} to $\mathcal P_n$ and $D_n$ for each $n$, we find a point-separating subspace in $C^2_p(X)$ that is a countable union of discrete subspaces. By Proposition \ref{pro:union}, $C^2_p(X)$ contains a discrete point-separating subspace.
\end{proof}

\par\bigskip\noindent
An argument identical to that of Theorem \ref{thm:criterionsigma} leads to the following statement for $C_p(X)$.
\par\bigskip\noindent
\begin{thm}\label{thm:criterionsigma0dim} Assume that $X$ is zero-dimensional. Then 
$C_p(X)$ has a discrete point-separating subspace if and only if $s(\sigma_X) \geq iw(X)$.
\end{thm}

\par\bigskip\noindent
Theorems \ref{thm:criterionsigma} and \ref{thm:criterionsigma0dim} imply the following.
\par\bigskip\noindent
\begin{cor}\label{CpXCpX2}
Let $X$ be a zero-dimensional space. Then $C_p(X)$ has a point-separating discrete subspaces if and only if $C^2_p(X)$ does.
\end{cor}

\par\bigskip\noindent
Note that the image of a point-separating family under a homeomorphism need not be point-separating. Indeed,  $\{id_{[0,1]}\}$ is a point-separating subspace of $C_p([0,1])$, However, one can construct an automorphism on $C_p([0,1])$ that caries $\{ id_{[0,1]}\}$ to $\{\vec 0\}$ which is not point-separating. In connection with this observation,
it would be interesting to know if having a discrete point-separating subspace is preserved by homeomorphisms among function spaces. The answer is affirmative and to show it we will use the fact \cite[I,1,6]{Arh} that if $C_p(X)$ and $C_p(Y)$ are homeomorphic then $iw(X)=iw(Y)$.

\par\bigskip\noindent
\begin{thm}\label{thm:tequiv0dim}
 Let $X$ and $Y$ be  $t$-equivalent.
If $C^2_p(X)$ has a discrete point-separating subspace, then so does $C^2_p(Y)$.
\end{thm}
\begin{proof}
Fix a homeomorphism $\phi: C^2_p(X)\to C^2_p(Y)$ and a discrete point-separating subspace $D$ of $C^2_p(X)$.

Assume, first, that $|D|$ is finite. Then $iw(X)=\omega$. Hence $iw(Y)=\omega$. Since $Y$ is infinite, it contains a an infinite countable subspace. By Theorem \ref{thm:criterionsigma} , $C^2_p(Y)$ contains a discrete point-separating subspace.

We now assume that $|D|$ is infinite. Then $|D|\geq iw(X)$. Therefore, $|\phi (D)|\geq iw(Y)$. By Lemma \ref{lem:functionspace0}, $s(\sigma_Y)\geq |\phi(D)|\geq iw(Y)$.  By Theorem \ref{thm:criterionsigma} , $C^2_p(Y)$ contains a discrete point-separating subspace.
\end{proof}

\par\bigskip\noindent
Repeating the argument of Theorem \ref{thm:tequiv0dim}, we obtain the following.
\par\bigskip\noindent
\begin{thm}\label{thm:tequiv} Let $X$ and $Y$ be zero-dimensional and $t$-equivalent.
If $C_p(X)$ has a discrete point-separating subspace, then so does $C_p(Y)$.
\end{thm}

\par\bigskip
While being a discrete subspace is already a nice property, it would be interesting to know when $C_p(X)$ or its finite power has a discrete point-separating subspace which is in addition  a subgroup. Note that any discrete subgroup is closed. In addition, $C_p(X)$ can be covered by countably many shifts of any neighborhood of zero-function. Therefore, any discrete subgroup of $C_p(X)$ is countable. These observations lead to the following question.

\par\bigskip\noindent
\begin{que}
Let $X$ be a separable metric space. Is it true that $C_p(X)$ has a discrete point-separating subgroup?
\end{que}

\par\bigskip\noindent
It is worth noting that  separable metric spaces  have many  pretty point-separating subspaces as backed up by the next two statement.

\par\bigskip\noindent
\begin{thm}\label{thm:01_1}
$C_p(X)$ has a point-separating subset homeomorphic to $[0,1]$ if and only $X$ admits a continuous injection into $\mathbb R^\omega$.
\end{thm}
\begin{proof} To prove necessity, let $F\subset C_p(X)$ a point-separating family homeomorphic to $[0,1]$. Then any dense subset of $F$ is point-separating too. Therefore, $C_p(X)$ has a countable point-separating family. Therefore, $X$ continuously injects into $\mathbb R^\omega$.

To prove sufficiency we need the following claim.
\par\medskip\noindent
{\it Claim. $\mathbb R^\omega$ embeds into $C_p([0,1])$.}
\par\noindent
To prove the claim, note that $C_p(\omega) = \mathbb R^\omega$ embeds into $C_p(\mathbb R)$ since $\omega$ is closed in $\mathbb R$. By Gulko-Hmyleva  theorem \cite{GH} that $\mathbb R$ and $[0,1]$ are $t$-equivalent, we conclude that, $\mathbb R^\omega$ embeds into $C_p([0,1])$. The claim is proved.
\par\medskip\noindent
By Claim $X$ injects  into $C_p([0,1])$. Let $F$ ne the image of such an injection. Due to homogeneity we may assume that the identity function is in $F$.Therefore, $F$ generates the topology of $[0,1]$. Consider the evaluation map  the evaluation function $\Psi_F:[0,1]\to C_p(F)$. Since $F$ generates the topology of $[0,1]$ , we conclude that $\Psi_F([0,1])$ genberate the topology of $F$. If $h:X\to F$ is a continuous bijection then the map $H: C_p(F)\to C_p(X)$ is a continuous injection, where $H(f) =hf$. Clearly, $H(\Psi_F([0,1]))=[0,1]$ is point separating.
\end{proof}

\par\bigskip\noindent
\begin{thm}\label{thm:01_2}
Let $X$ be a separable metric space. Then $C_p(X)$ has a topology-generating subspace homeomorphic to $[0,1]$.
\end{thm}
\begin{proof}
Embed $X$ into $C_p([0,1])$ so that the image $F$ contains the identity map. 
The evaluation function $\Psi_F:[0,1]\to C_p(F)$. Since $F$ generates the topology of $[0,1]$ and therefore $\Psi_F([0,1])$ generates the topology of $F$. Since $F=X$, we conclude that $[0,1]=\Psi_F([0,1])$ generates the topology of  $F=X$. 
\end{proof}

\par\bigskip\noindent
Note that Theorem \ref{thm:01_2} cannot be reversed. Indeed, $[0,1]$ generates the topology of $C_p([0,1])$ but the latter is not metrizable.

\par\bigskip
We would like to finish with two problems that are naturally prompted by our study.
\par\bigskip\noindent
\begin{que}
Characterize spaces $X$ for which $C_p(X)$ has a closed discrete point-separating subset.
\end{que}
\par\bigskip\noindent
\begin{que}
Characterize spaces $X$ for which $C_p(X)$ has a (closed) discrete topology-generating subset.
\end{que}

\par\bigskip\noindent
At last, the unattained goal of the paper is left as the following question.
\par\bigskip\noindent
\begin{que}
Assume that $C_p(X)$ has a discrete subspace of size $iw(X)$. Is it true that $C_p(X)$ has a discrete point-separating set?
\end{que}


\begin{thebibliography}{99}

\bibitem{Arh}
A. Arhangelskii, {\it Topological Function Spaces}, Math. Appl., vol. 78, Kluwer Academic
Publishers, Dordrecht, 1992.

\bibitem{Eng}
R. Engelking, {\it General Topology}, PWN, Warszawa, 1977.

\bibitem{Fed}
V. V. Fedorchuk, {\it A compact having a cardinality of continuum with no convergent sequences}, Math. Proc. Cambridge Phil. Soc. 81(1977), 177-181

\bibitem{GH} 
S. Gulko and T. Hmyleva, {\it Compactaness is not preserved by $t$-equivalence}, Mat Zametki, vol 39, 6 (1986), 895-903.




\bibitem{Iva}
A. V. Ivanov, {\it On bicompacta with hereditary separable finite powers}, (in Russian)  DAN SSSR, 243 (1978), 1109-1112.

\bibitem{Mar}
S. Mardesic,  {\it On covering dimension and inverse limits of compact spaces}, Ill. J. of Math. 4 (1960), 278-291.



\bibitem{Zen}
P. Zenor, {\it Hereditary m-separability and the hereditary m-Lindel\"of property in product spaces and function spaces},
Fund. Math. 106 (1980), 175-80.




\end{thebibliography}
\end{document}